\documentclass[11pt,reqno,oneside]{amsart}
\usepackage{amscd}
\usepackage{amsmath}
\usepackage{latexsym}
\usepackage{amsfonts}
\usepackage{amssymb}
\usepackage{amsthm}
\usepackage{graphicx}
\usepackage{verbatim}
\usepackage{mathrsfs}
\usepackage{enumerate}
\usepackage{hyperref}
\usepackage{fullpage}
\usepackage{xcolor}
\usepackage{pb-diagram}

\theoremstyle{plain}
\theoremstyle{plain}\newtheorem{theorem}{Theorem}[section]
\theoremstyle{plain}\newtheorem{lemma}[theorem]{Lemma}
\theoremstyle{plain}
\theoremstyle{plain}\newtheorem{proposition}[theorem]{Proposition}
\theoremstyle{plain}

\numberwithin{equation}{section}

\newcommand{\be}{\begin{equation}}
\newcommand{\ee}{\end{equation}}
 \newcommand{\ba}{\begin{aligned}}
 \newcommand{\ea}{\end{aligned}}

  \newcommand{\ben}{\begin{enumerate}}
   \newcommand{\een}{\end{enumerate}}

\newcommand{\Rmnum}[1]{\expandafter\@slowromancap\romannumeral #1@}
\allowdisplaybreaks

\begin{document}

\title{Ko{\l}odziej-Tosatti's conjecture on compact Hermitian manifold with bounded mass property}

\author[]{Lei Zhang}

\address{Yau Mathematical Sciences Center, Tsinghua University, Beijing, 100084, P.R.China}
\email{leizhang92@mail.tsinghua.edu.cn}

\subjclass[2010]{}

\begin{abstract}
In this note, we show a conjecture of Ko{\l}odziej-Tosatti about Morse-type integrals in nef $(1,1)$ classes on compact Hermitian manifold with bounded mass property. As a consequence, we give positive answers to Demailly-P\u{a}un's conjecture and Tosatti-Weinkove's conjecture when compact Hermitian manifold with bounded mass property.
\end{abstract}
\smallskip
\maketitle

\section{Introduction}
Let $(X^n,\omega)$ be a compact Hermitian manifold and $\alpha$ a closed real $(1,1)$ form on $X$. A classes $[\alpha]\in H^{1,1}_{BC}(X;\mathbb{R})$ is called nef (numerical effective) if for every $\varepsilon>0$ there exists $u_{\varepsilon}\in C^{\infty}(X,\mathbb{R})$ such that $$\alpha+\sqrt{-1}\partial\bar{\partial}u_{\varepsilon}\ge-\varepsilon\omega.$$
A classes $[\alpha]\in H^{1,1}_{BC}(X;\mathbb{R})$ is called big if for every $\varepsilon>0$ there exists a quasi-psh function $u$ on $X$ (locally the sum of psh plus smooth) such that $$T:=\alpha+\sqrt{-1}\partial\bar{\partial}u\ge\varepsilon\omega$$ holds in the current sense. In this case, Demailly-P\u{a}un \cite[Theorem 2.12]{DemPau04} show that $X$ must be bimeromorphic to a compact K\"ahler manifold.

Our main interest is in a version of conjecture of Demailly \cite{Dem10} for nef classes on non-K\"ahler manifolds, which also encompasses a question proposed by Tosatti in \cite[Remark 3.2]{Tos16}. More precisely, Ko{\l}odziej-Tosatti \cite{KolTos21} proposed the following conjecture.

\textbf{Conjecture 1:}(Ko{\l}odziej-Tosatti, \cite[Conjecture 1.2]{KolTos21}) Let $(X^n,\omega)$ be a compact Hermitian manifold and $\alpha$ a closed real $(1,1)$ form such that $[\alpha]$ is nef. Then we have
\begin{equation}
\int_X\alpha^n=\inf_{u\in C^{\infty}(X,\mathbb{R})}\int_{X(\alpha+\sqrt{-1}\partial\bar{\partial}u,0)}\big(\alpha+\sqrt{-1}\partial\bar{\partial}u\big)^n
\end{equation}
where $X(\alpha+\sqrt{-1}\partial\bar{\partial}u,0)$ denotes the set of all points $x\in X$ such that $(\alpha+\sqrt{-1}\partial\bar{\partial}u)\ge0.$

In \cite{Dem10} Demailly shows that the inequality
$$\int_X\alpha^n\leq\inf_{u\in C^{\infty}(X,\mathbb{R})}\int_{X(\alpha+\sqrt{-1}\partial\bar{\partial}u,0)}\big(\alpha+\sqrt{-1}\partial\bar{\partial}u\big)^n.
$$
holds (see Step 2 in the proof of Theorem \ref{Kolo-To-ture}) and the inequality holds if $X$ is projective surface and $\alpha\in NS_{\mathbb{R}}(X)$ (\cite[Theorem 1.18]{Dem10}). The K\"ahler version of conjecture 1 was shown by Demailly \cite[Sec 3.5]{Dem11} whenever the orthogonality conjecture of Zariski decompositions of BDPP \cite{BDPP13} holds and has been proved by Ko{\l}odziej-Tosatti \cite[Proposition 2.2]{KolTos21} when $\alpha$ is big or more generally in Fujiki class. In the non-K\"ahler case, the conjecture 1 holds \cite[Proposition 2.1]{KolTos21}if either the class $[\alpha]$ is semi-positive or the manifold $X$ admits a Hermitian metric $\omega$ with
\begin{equation}\label{pluriclosed+}
\partial\bar{\partial}\omega=0=\partial\bar{\partial}\omega^2.
\end{equation}

Boucksom \cite{Bou02} defines the volume of $[\alpha]$ to be
$$vol([\alpha]):=\sup_T\int_XT_{ac}^n>0,$$
where the supermum is over all K\"ahler currents in the class $[\alpha]$ and $T_{ac}$ is the absolutely continuous part of the Lebesgue decomposition with respect to the Lebesgue measure on $X$. If $[\alpha]$ is not big, then one can define $vol([\alpha])=0.$ Boucksom \cite[Conjecture]{Bou02} conjectured that there exists a K\"ahler current in $[\alpha]\in H^{1,1}_{BC}$ on a compact Hermitian manifold $(X,\omega)$ provided that $\int_X\alpha^n>0.$ Boucksom \cite[Theorem 4.7]{Bou02} has proved his conjecture when $\omega$ is K\"ahler. In the non-K\"ahler manifold, Wang \cite{Wan16,Wan19} has produced partial results on Boucksom's conjecture when $[\alpha]$ is pseudo-effective real $(1,1)$-class. Recently, Boucksom-Guedj-Lu \cite{BGL24+} shown that the lower volume function vanishes outside the big cone (\cite[Theorem 3.20]{BGL24+}). As a simple consequence, they proved that $X$ is Fujiki if $X$ has the bounded mass property and it admits a closed positive $(1,1)$-current $T$ such that $\int_XT^n>0$ (that is, $\int_X\alpha^n>0$, see\cite{Bou02,BGL24+} for more details).

As was observed by Tosatti \cite{Tos16}, conjecture 1 also has applications to complex Monge-Amp\`ere equations on non-K\"ahler manifolds. Inspired by the work of Demailly-P\u{a}un \cite{DemPau04}, we see that the following conjecture of Demailly-P\u{a}un can be obtained by applying the Demailly-P\u{a}un's methods. Based on the work of Demailly \cite{Dem93} and Tosatti-Weinkove \cite{TosWei12}, it is interesting to attack Tosatti-Weinkove's conjecture.

\textbf{Conjecture 2:}(Demailly-P\u{a}un, \cite[Conjecture 0.8]{DemPau04}) Let $(X^n,\omega)$ be a compact Hermitian manifold and $\alpha$ a closed real $(1,1)$ form such that $[\alpha]$ is nef and $\int_X\alpha^n>0$. Then $[\alpha]$ is big.

The K\"ahler version of conjecture 2 was initially proved by Demailly-P\u{a}un through Demailly's mass concentration technique \cite{Dem93} and Yau's solution to Calabi conjecture \cite{Yau78}. In the case when $n=2$ and $\alpha$ is semi-positive, conjecture 2 was proved by the work of Buchdahl \cite{Buc99,Buc00} and Lamari \cite{Lam99}. For the case when $n=3,$ Chiose \cite[Theorem 4.1]{Chi16} has proved conjecture 2 under the assumption that there exists a pluriclosed Hermitian metric $\omega$ ($\partial\bar{\partial}\omega=0$). Later on, Popovici \cite[Theorem 1.1]{Pop16} made an observation (based on Xiao's approach \cite[Theorem 1.4]{Xiao}) that can help to simplify the some arguments in \cite{Chi16} and proved conjecture 2 under the assumption (\ref{pluriclosed+}). Nguyen \cite[Theorem 0.7]{Ngu16} used Popovici's observation and proved conjecture 2 assuming that $\alpha$ is semi-positive and that there exists a pluriclosed Hermitian metric. Recently, Li-Wang-Zhou \cite[Theorem 1.11]{LWZ24} relaxed the semi-positivity assumption of $\alpha$ to the assumption that there exists a bounded quasi-psh function $\rho$ such that $\alpha+\sqrt{-1}\partial\bar{\partial}\rho\ge0$ in the current sense. Guedj-Lu \cite{GueLu22,GueLu23} made a breakthrough by introducing bound mass property for the Hermitian metric (\cite{BGL24+}), then conjecture is valid. It is worth noting that all the aforementioned results can also be regarded as partial answers to a conjecture of Boucksom \cite{Bou02}. Recent advances towards a conjecture of Boucksom were recently made by Boucksom-Guedj-Lu \cite[Theorem C]{BGL24+}, they proved the Boucksom's conjecture if it has the bounded mass property.

\textbf{Conjecture 3:}(Tosatti-Weinkove, \cite{TosWei12}) Let $(X^n,\omega)$ be a compact Hermitian manifold and $\alpha$ a closed real $(1,1)$ form such that $[\alpha]$ is nef and $\int_X\alpha^n>0$. Fix $x_1,\cdots,x_m\in X$ and choose positive real numbers $\tau_i,\cdots,\tau_m$ such that
$$\sum_{i=1}^m\tau_i^n<\int_X\alpha^n.$$
Then there exists a $\alpha$-psh function $\phi$ with logarithmic poles at $x_1,\cdots,x_m:$
$$\phi(z)\le\tau_i\log|z|+O(1),$$
in a coordinate $(z_1,\cdots,z_n)$ at $x_i.$

In his seminal paper \cite{Dem93}, Demailly used a mass concentration technique for a degenerating family of complex Monge-Amp\`ere equations to produce $\alpha$-psh in a nef and big class $[\alpha]$ on a compact K\"ahler manifold which have nontrivial Lelong numbers at finitely many given points. This was later extended by Tosatti-Weinkove \cite[Main Theorem]{TosWei12} to non-K\"ahler manifolds and proved their conjecture for $n=2,3$ and they proved for general $n\ge4$ under the conditions that $X$ is Moishezon and $\alpha$ is rational. Significant advancements have been made more recently in this field. Tosatti \cite[Theorem 1.3]{Tos16} proved conjecture 3 when $\alpha$ is nef and big. Nguyen \cite[Theorem 0.6]{Ngu16} proved conjecture 3 under the assumption that $\alpha$ is semi-positive and $\int_X\alpha^n>0.$ Recently, Li-Wang-Zhou \cite[Theorem 1.10]{LWZ24} relaxed the semi-positivity assumption of $\alpha$ to the assumption that there exists a bounded quasi-psh function $\rho$ such that $\alpha+\sqrt{-1}\partial\bar{\partial}\rho\ge0$ in the current sense.

Now, we recall that a compact Hermitian manifold $(X,\omega)$ has bounded mass property if it satisfies
$$\overline{vol}(\omega):=\sup\big\{\int_X(\omega+\sqrt{-1}\partial\bar{\partial}f)^n;~~f\in C^{\infty}(X;\mathbb{R})~with~\omega+\sqrt{-1}\partial\bar{\partial}f>0\big\}<\infty.$$

The condition $\overline{vol}(\omega)<\infty$ is independent of the chioce of $\omega;$ it is moreover invariant under bimeromorphic change of coordinates \cite[Theoerm A]{GueLu22}. In particular, $\overline{vol}(\omega)<\infty$ if $X$ belongs to the Fujiki class. In recent years, the bounded mass property has been intensively studied by many authors, in relation to quasi-psh envelopes (see \cite{BGL24+,GueLu22,GueLu23} and the reference therein).

In this note, we show the following result.
\begin{theorem}\label{Kolo-To-ture}
The conjecture 1 is true if it has the bounded mass property.
\end{theorem}

\begin{theorem}\label{DePa-ture}
The conjecture 2 is true if it has the bounded mass property. In particular, the conjecture 3 is true if it has the bounded mass property.
\end{theorem}

\section{Proof of the main result}
\subsection{Proof of Theorem \ref{Kolo-To-ture}} We now present a proof of this result following the argument of Ko{\l}odziej-Tosatti \cite{KolTos21}. The proof is divided into two steps. Let $(n,k)=\frac{n!}{k!(n-k)!}.$\\
\textbf{Step 1.} Since $[\alpha]$ is nef, then for every $\varepsilon>0$ there exists $u_{\varepsilon}\in C^{\infty}(X,\mathbb{R})$ such that $$\alpha+\sqrt{-1}\partial\bar{\partial}u_{\varepsilon}\ge-\varepsilon\omega.$$
So, $$X(\alpha+\sqrt{-1}\partial\bar{\partial}u_{\varepsilon}+\varepsilon\omega,0)=X.$$
By rescaling, we can assume that $\alpha\le\omega.$ Therefore, for any smooth function $$u\in psh(X,\alpha+\varepsilon\omega)\subset psh(X,2\omega)$$
we apply the bounded mass property and Stokes' theorem to have
\begin{align}\label{alpha+e-ub1}
\nonumber&\inf_{u\in C^{\infty}(X,\mathbb{R})}\int_{X(\alpha+\varepsilon\omega+\sqrt{-1}\partial\bar{\partial}u,0)}\big(\alpha+\varepsilon\omega+\sqrt{-1}\partial\bar{\partial}u\big)^n\\
\nonumber
\le&\int_X\big(\alpha+\varepsilon\omega+\sqrt{-1}\partial\bar{\partial}u_{\varepsilon}\big)^n\\ \nonumber
=&\int_X\big(\alpha+\sqrt{-1}\partial\bar{\partial}u_{\varepsilon}\big)^n+\sum_{k=1}^n(n,k)\varepsilon^k\int_X\big(\alpha+\sqrt{-1}\partial\bar{\partial}u_{\varepsilon}\big)^{n-k}\wedge\omega^k\\ \nonumber
=&\int_X\alpha^n+\sum_{k=1}^n(n,k)\varepsilon^k\int_X\big(\alpha-2\omega+2\omega+\sqrt{-1}\partial\bar{\partial}u_{\varepsilon}\big)^{n-k}\wedge\omega^k\\ \nonumber
\leq&\int_X\alpha^n+\sum_{k=1}^n(n,k)\varepsilon^k\int_X\big(2\omega+\sqrt{-1}\partial\bar{\partial}u_{\varepsilon}\big)^{n-k}\wedge\omega^k\\
=&\int_X\alpha^n+O(\varepsilon).
\end{align}
On the other hand, given any $u\in C^{\infty}(X,\mathbb{R}),$ we have
$$X(\alpha+\sqrt{-1}\partial\bar{\partial}u,0)\subset X(\alpha+\varepsilon\omega+\sqrt{-1}\partial\bar{\partial}u,0)$$
and for any $x\in X(\alpha+\sqrt{-1}\partial\bar{\partial}u,0)$
$$\big(\alpha+\sqrt{-1}\partial\bar{\partial}u\big)^n(x)\le\big(\alpha+\varepsilon\omega+\sqrt{-1}\partial\bar{\partial}u\big)^n(x).$$
So, combining (\ref{alpha+e-ub1}) with above, we have
\begin{align}\label{alpha+e-ub2}
\nonumber&\inf_{u\in C^{\infty}(X,\mathbb{R})}\int_{X(\alpha+\sqrt{-1}\partial\bar{\partial}u,0)}\big(\alpha+\sqrt{-1}\partial\bar{\partial}u\big)^n\\
\nonumber
\le&\inf_{u\in C^{\infty}(X,\mathbb{R})}\int_{X(\alpha+\varepsilon\omega+\sqrt{-1}\partial\bar{\partial}u,0)}\big(\alpha+\varepsilon\omega+\sqrt{-1}\partial\bar{\partial}u\big)^n\\
\le&\int_X\alpha^n+O(\varepsilon).
\end{align}
Letting $\varepsilon\to0,$ we have
\begin{equation}\label{alpha+energy-lb}
\int_X\alpha^n\ge \inf_{u\in C^{\infty}(X,\mathbb{R})}\int_{X(\alpha+\sqrt{-1}\partial\bar{\partial}u,0)}\big(\alpha+\sqrt{-1}\partial\bar{\partial}u\big)^n.
\end{equation}
\textbf{Step 2.} To prove the converse inequality:
\begin{equation}\label{alpha+energy-ub}
\int_X\alpha^n\le \inf_{u\in C^{\infty}(X,\mathbb{R})}\int_{X(\alpha+\sqrt{-1}\partial\bar{\partial}u,0)}\big(\alpha+\sqrt{-1}\partial\bar{\partial}u\big)^n.
\end{equation}
Fix $u\in C^{\infty}(X,\mathbb{R}),$ let $\beta=\alpha+\sqrt{-1}\partial\bar{\partial}u.$ We consider the envelpope
\begin{align*}
h_{\varepsilon}(x):&=\sup\{\varphi(x)\in psh(X,\beta+\varepsilon\omega),\varphi\le0\}\\
&=-u+u_{\varepsilon}+\sup\{\varphi(x)\in psh(X,\alpha+\varepsilon\omega+\sqrt{-1}\partial\bar{\partial}u_{\varepsilon}),\varphi\le u-u_{\varepsilon}\}
\end{align*}
which satisfies $h_{\varepsilon}\in C^{1,1}(X)$ (See \cite{ChuZ19,Tos18}). In particular,
$$\big(\beta+\varepsilon\omega+\sqrt{-1}\partial\bar{\partial}h_{\varepsilon}\big)^n=0,~~~on~~~\{h_{\varepsilon}<0\}.$$
Hence, applying above inequality and Proposition 3.1 (iii) in \cite{Ber09}, we have
\begin{align}\label{alpha+e-lb1}
\nonumber\int_X\big(\beta+\varepsilon\omega+\sqrt{-1}\partial\bar{\partial}h_{\varepsilon}\big)^n&=\int_{h_{\varepsilon}=0}\big(\beta+\varepsilon\omega\big)^n\\
&\le\int_{X(\beta+\varepsilon\omega,0)}\big(\beta+\varepsilon\omega\big)^n.
\end{align}
Noting that
\begin{align}\label{alpha+e-lb2}
\lim_{\varepsilon\to0}\int_{X(\beta+\varepsilon\omega,0)}\big(\beta+\varepsilon\omega\big)^n=\int_{X(\beta,0)}\beta^n
\end{align}
Next, we show that
\begin{align}\label{alpha+e-lb3}
\limsup_{\varepsilon\to0}\int_X\big(\beta+\varepsilon\omega+\sqrt{-1}\partial\bar{\partial}h_{\varepsilon}\big)^n\ge\int_X\alpha^n.
\end{align}
Indeed, similar calculation of (\ref{alpha+e-ub1}), we have
\begin{align*}
\nonumber\int_X\big(\beta+\varepsilon\omega+\sqrt{-1}\partial\bar{\partial}h_{\varepsilon}\big)^n
=&\int_X\big(\alpha+\varepsilon\omega+\sqrt{-1}\partial\bar{\partial}(h_{\varepsilon}+u)\big)^n\\
=&\int_X\alpha^n+O(\varepsilon).
\end{align*}
Combining (\ref{alpha+e-lb1}), (\ref{alpha+e-lb2}) with (\ref{alpha+e-lb3}), by letting $\varepsilon\to0,$ the inequality (\ref{alpha+energy-ub}) is proved. From above, this completes the proof.\qed

\subsection{Proof of Theorem \ref{DePa-ture}}
We can now give the proof of Theorem \ref{DePa-ture}. We will follow closely Popovici's arguments \cite{Pop16} (see \cite[p.398-399]{Tos16}). The starting point for our discussion is the following lemma by Lamari in \cite[Lemma 3.3]{Lam99}(see \cite[Lemma 3.3]{Tos16} simple proof).
\begin{lemma}[Lamari, Lemma 3.3 in \cite{Lam99}]\label{lamari99}
Let $(X^n,\omega)$ be a compact Hermitian manifold and $\alpha$ a real $(1,1)$ form on $X$. Given $\varepsilon\geq0,$ there exists a current of the form $T=\alpha+\sqrt{-1}\partial\bar{\partial}u\geq\varepsilon\omega$ if and only if
$$\int_X\alpha\wedge\omega_G^{n-1}\geq\varepsilon\int_X\omega\wedge\omega_G^{n-1}$$
holds for all Gauduchon metrics $\omega_G$ on $X$.
\end{lemma}

Since $[\alpha]$ is nef, then for every $\varepsilon>0$ there exists $u_{\varepsilon}\in C^{\infty}(X,\mathbb{R})$ such that $$\alpha_{\varepsilon}:=\alpha+\varepsilon\omega+\sqrt{-1}\partial\bar{\partial}u_{\varepsilon}>0.$$
By Lamari's Lemma \ref{lamari99}, it suffices to show that there exist $\varepsilon,\delta>0$ such that for any Gauduchon metric $\omega_G$ on $X$ we have
\begin{equation}\label{lamari-inequality}
\int_X\alpha_{\varepsilon}\wedge\omega_G^{n-1}\geq\varepsilon(1+\delta)\int_X\omega\wedge\omega_G^{n-1}.
\end{equation}
In the following, we divide the proof of (\ref{lamari-inequality}) into three steps.\\
\textbf{Step 1.} Thanks to Tosatti-Weinkove \cite{TosWei10}, we can find $\psi_{\varepsilon}\in C^{\infty}(X;\mathbb{R})$ solving
\begin{equation}\label{cma+}
  \begin{cases}
  \big(\alpha_{\varepsilon}+\sqrt{-1}\partial\bar{\partial}\psi_{\varepsilon}\big)^n&=e^{b_{\varepsilon}}\frac{\omega\wedge\omega_G^{n-1}}{\int_X\omega\wedge\omega_G^{n-1}},\\
  \widetilde{\alpha_{\varepsilon}}&=\alpha_{\varepsilon}+\sqrt{-1}\partial\bar{\partial}\psi_{\varepsilon}
  \end{cases}
\end{equation}
where $b_{\varepsilon}\in\mathbb{R}.$ Noting that
\begin{align*}
e^{b_{\varepsilon}}=\int_X\widetilde{\alpha_{\varepsilon}}^n&=e^{\frac{b_{\varepsilon}^n}{2}}\Big(\int_X\big(\frac{\widetilde{\alpha_{\varepsilon}}^n}{\omega_G^n}\big)^{\frac{1}{2}}\big(\frac{\omega\wedge\omega_G^{n-1}}{\omega_G^n}\big)^{\frac{1}{2}}\omega_G^n\Big)\big(\int_X\omega\wedge\omega_G^{n-1}\big)^{-\frac{1}{2}}\\
&=\frac{e^{\frac{b_{\varepsilon}}{2}}}{\sqrt{n}}\Big(\int_X\big(\frac{\widetilde{\alpha_{\varepsilon}}^n}{\omega_G^n}\big)^{\frac{1}{2}}\big(tr_{\omega_G}\omega\big)^{\frac{1}{2}}\omega_G^n\Big)\big(\int_X\omega\wedge\omega_G^{n-1}\big)^{-\frac{1}{2}}.
\end{align*}
Hence,
\begin{align*}
e^{b_{\varepsilon}}&=\frac{1}{n}\Big(\int_X\big(\frac{\widetilde{\alpha_{\varepsilon}}^n}{\omega_G^n}\big)^{\frac{1}{2}}\big(tr_{\omega_G}\omega\big)^{\frac{1}{2}}\omega_G^n\Big)^2\big(\int_X\omega\wedge\omega_G^{n-1}\big)^{-1}\\
&\le\frac{1}{n}\Big(\int_X\big(\frac{\widetilde{\alpha_{\varepsilon}}^n}{\omega_G^n}tr_{\widetilde{\alpha_{\varepsilon}}}\omega\big)^{\frac{1}{2}}\big(tr_{\omega_G}\widetilde{\alpha_{\varepsilon}}\big)^{\frac{1}{2}}\omega_G^n\Big)^2\big(\int_X\omega\wedge\omega_G^{n-1}\big)^{-1}\\
&\le\frac{1}{n}\Big(\int_X\big(tr_{\widetilde{\alpha_{\varepsilon}}}\omega\big)\widetilde{\alpha_{\varepsilon}}^n\Big)\Big(\int_X\big(tr_{\omega_G}\widetilde{\alpha_{\varepsilon}}\big)\omega_G^n\Big)\big(\int_X\omega\wedge\omega_G^{n-1}\big)^{-1}\\
&=n\Big(\int_X\omega\wedge\widetilde{\alpha_{\varepsilon}}^{n-1}\big)\Big(\int_X\widetilde{\alpha_{\varepsilon}}\wedge\omega_G^{n-1}\Big)\big(\int_X\omega\wedge\omega_G^{n-1}\big)^{-1}
\end{align*}
that is,
\begin{align}
\frac{\int_X\alpha_{\varepsilon}\wedge\omega_G^{n-1}}{\varepsilon\int_X\omega\wedge\omega_G^{n-1}}=\frac{\int_X\widetilde{\alpha_{\varepsilon}}\wedge\omega_G^{n-1}}{\varepsilon\int_X\omega\wedge\omega_G^{n-1}}\geq\frac{\int_X\widetilde{\alpha_{\varepsilon}}^n}{n\varepsilon\int_X\omega\wedge\widetilde{\alpha_{\varepsilon}}^{n-1}}
\end{align}
\textbf{Step 2.} There exist $\varepsilon,\delta'>0$ independent of $\omega_G$ such that
\begin{equation}\label{lamari-2}
n\varepsilon\int_X\omega\wedge\widetilde{\alpha_{\varepsilon}}^{n-1}\le (1-\delta')\int_X\widetilde{\alpha_{\varepsilon}}^n.
\end{equation}
Indeed, similar calculation of (\ref{alpha+e-ub1}), we have
\begin{align}\label{alpha-e1}
\nonumber\int_X\widetilde{\alpha_{\varepsilon}}^n&=\int_X\big(\alpha+\sqrt{-1}\partial\bar{\partial}(u_{\varepsilon}+\psi_{\varepsilon})+\varepsilon\omega\big)^n\\
&=...=\int_X\alpha^n+O(\varepsilon).
\end{align}
On the other hand, by the bounded mass property and Stokes' formula again, we have
\begin{align}\label{alpha-e2}
\nonumber n\varepsilon\int_X\omega\wedge\widetilde{\alpha_{\varepsilon}}^{n-1}&=n\varepsilon\int_X\omega\wedge\big(\alpha+\varepsilon\omega+\sqrt{-1}\partial\bar{\partial}(u_{\varepsilon}+\psi_{\varepsilon})\big)^{n-1}\\
\nonumber&=n\varepsilon\int_X\omega_X\wedge\alpha^{n-1}+n\varepsilon\sum_{k=1}^{n-2}\int_X\omega\wedge^{n-k-1}\wedge\big(\varepsilon\omega+\sqrt{-1}\partial\bar{\partial}(u_{\varepsilon}+\psi_{\varepsilon})\big)^k\\
&=n\varepsilon\int_X\omega_X\wedge\alpha^{n-1}+O(\varepsilon^2).
\end{align}
Case 1: $\int_X\omega_X\wedge\alpha^{n-1}\le0$. By (\ref{alpha-e1}) and (\ref{alpha-e2}), (\ref{lamari-2}) is obvious.\\
Case 2: $\int_X\omega_X\wedge\alpha^{n-1}\ge0$. If we choose $0<\varepsilon\ll\frac{\int_X\alpha^n}{2n\int_X\omega\wedge\alpha^{n-1}},$ then by (\ref{alpha-e1}) and (\ref{alpha-e2}), (\ref{lamari-2}) is valid. From above, the inequality (\ref{lamari-inequality}) is proved. \\
\textbf{Step 3.} To prove conjecture 3. Indeed, by above, we see that if $\alpha$ is nef and $\int_X\alpha^n>0,$ then $\alpha$ is big. Hence, applying \cite[Theorem 1.3]{Tos16}, we finish the conjecture 3. For more details, we refer the reader's see \cite{Dem93,TosWei12,Tos16} and the reference therein.

From above, this completes the proof of conjecture 2 and conjecture 3. \qed

\section{Further Discussion}
As remarked above, we will follow closely Demailly-P\u{a}un's arguments \cite{DemPau04} and Tosatti-Weinkove's arguments \cite{TosWei12} ( cf.\cite{Tos16}) to give an alternative proof of Theorem \ref{DePa-ture}. For reader's convenience, we recall Demailly-P\u{a}un's concentration of mass for nef classes of positive self-intersection. 
\begin{lemma}(Demailly-P\u{a}un, \cite{DemPau04}, Boucksom, \cite{Bou02})\label{DP-mass-cen}
Let $(X^n,\omega)$ be a compact Hermitian manifold and let $Y\subset X$ be an analytic subset of $X$. Then there exist globally defined quasi-psh function $\psi$ with analytic singularities along $Y$ and a sequence of smooth functions $\big(\psi_{\varepsilon}\big)_{\varepsilon\in(0,1]}$ on $X$ decreasing pointwise to $\psi$ such that
\begin{equation}\label{mass-cen-0}
\omega_{\varepsilon}:=\omega+\sqrt{-1}\partial\bar{\partial}\psi_{\varepsilon}\geq\frac{1}{2}\omega
\end{equation}
and
\begin{equation}\label{mass-cen-1}
\int_{U\cap V_{\varepsilon}}\omega_{\varepsilon}^p\wedge\omega^{n-p}\geq\delta_p(U)>0
\end{equation}
in any neighborhood $U$ of a regular point $x_0\in Y,$ where $V_{\varepsilon}=\{z\in X;~\psi(z)<\log\varepsilon\}$ is the "tubular neighborhood" of radius $\varepsilon$ around $Y$. In particular, for every  there exists a closed positive current $T\in\{\omega_{\varepsilon}^p\}$ (any weak limit $T$ of $\omega_{\varepsilon}^p$ as $\varepsilon\to0$) satisfies $T\geq\delta'[Y]$ for some $\delta'>0$, that is,
$\int_YT\ge \delta'.$
\end{lemma}
\begin{proposition}(=Conjecture 2)\label{DP-conj}
Let $(X^n,\omega)$ be a compact Hermitian manifold with bounded mass property and $\alpha$ a closed real $(1,1)$ form such that $[\alpha]$ is nef and $\int_X\alpha^n>0$. Then $[\alpha]$ is big. In particular, the conjecture 3 is true.
\end{proposition}
\begin{proof}
The proof of Proposition \ref{DP-conj} is divided into three steps.\\
\textbf{Step 1.} In the same setup as above, consider the following complex Monge-Amp\`ere equations
\begin{equation}\label{cmaes}
\big(\alpha+\varepsilon\omega+\sqrt{-1}\partial\bar{\partial}(u_{\varepsilon}+\phi_{\varepsilon})\big)^n=C_{\varepsilon}\omega_{\varepsilon}^n
\end{equation}
with
$$\widetilde{\alpha_{\varepsilon}}:=\alpha+\varepsilon\omega+\sqrt{-1}\partial\bar{\partial}u_{\varepsilon}+\sqrt{-1}\partial\bar{\partial}\phi_{\varepsilon}>0,~~\sup_X(u_{\varepsilon}+\phi_{\varepsilon})=0,$$
where for every $\varepsilon>0,$ and some uniquely determined positive constants $C_{\varepsilon}$ and $\omega_{\varepsilon}$ is a family metrics construct by Demailly-P\u{a}un \cite{DemPau04}. In the K\"ahler case, we have
$$C_{\varepsilon}=\frac{\int_X\widetilde{\alpha_{\varepsilon}}^n}{\int_X\omega_{\varepsilon}^n}=\frac{\int_X(\alpha+\varepsilon\omega)^n}{\int_X\omega^n}\ge C_0=\frac{\int_X\alpha^n}{\int_X\omega^n}>0.$$
In the non-K\"ahler manifolds, by the bounded mass property again,
$$C_{\varepsilon}=\frac{\int_X\widetilde{\alpha_{\varepsilon}}^n}{\int_X\omega_{\varepsilon}^n}=\frac{\int_X\alpha^n+O(\varepsilon)}{\int_X\omega^n+O(\varepsilon)}=C_0>0.$$
\textbf{Step 2.} Let us denote by
$$\lambda_1(z)\le\cdots\le\lambda_n(z)$$
the eigenvalues of $\widetilde{\alpha_{\varepsilon}}(z)$ with respect to $\omega_{\varepsilon}(z)$, at every point $z\in X$ (these functions are continuous with respect to $z$ and of course depend also on $\varepsilon$). The equation (\ref{cmaes}) is equivalent to the fact that
$$\lambda_1(z)\cdots\lambda_n(z)=C_{\varepsilon}$$
is constant (which is bounded away from $0$).

Fix a regular point $x_0\in Y$ and a small neighborhood $U$ (meeting only the irreducible component of $x_0$ in $Y$). By Lemma \ref{DP-mass-cen}, we have a uniform lower bound
\begin{equation}\label{mass-cen-2}
\int_{U\cap V_{\varepsilon}}\omega_{\varepsilon}^p\wedge\omega^{n-p}\geq\delta_p(U)>0.
\end{equation}
Now, by looking at the $p$ smallest (resp. $(n-p)$ largest) eigenvalues $\lambda_j$ of $\widetilde{\alpha_{\varepsilon}}(z)$ with respect to $\omega_{\varepsilon}(z)$,
\begin{align}
\widetilde{\alpha_{\varepsilon}}^p&\geq\lambda_1\cdots\lambda_p\omega_{\varepsilon}^p, \label{mass-cen-3}\\
\widetilde{\alpha_{\varepsilon}}^{n-p}\wedge\omega_{\varepsilon}^p&\ge\frac{1}{n!}\lambda_{p+1}\cdots\lambda_n\omega_{\varepsilon}^n.\label{mass-cen-4}
\end{align}
The inequality (\ref{mass-cen-4}) and bound mass property implies
\begin{align*}
\int_X\lambda_{p+1}\cdots\lambda_n\omega_{\varepsilon}^n&\le n!\int_X\widetilde{\alpha_{\varepsilon}}^{n-p}\wedge\omega_{\varepsilon}^p\\
&\le n!\int_X\alpha^{n-p}\wedge\omega^p+O(\varepsilon)\leq M
\end{align*}
for some constant $M$ when $\varepsilon<<1.$ In particular, for every $\delta>0$ the subset $E_{\delta}\subset X$ of points $z\in X$ such that
$$\lambda_{p+1}(z)\cdots\lambda_n(z)>\frac{M}{\delta}.$$
Hence, combining above inequality with (\ref{mass-cen-0})
\begin{equation}\label{mass-cen-5}
\int_{E_{\delta}}\omega_{\varepsilon}^p\wedge\omega^{n-p}\le 2^{n-p}\int_{E_{\delta}}\omega_{\varepsilon}^n\leq 2^{n-p}\delta.
\end{equation}
The combination of (\ref{mass-cen-2}) and (\ref{mass-cen-5}) yields
$$\int_{(U\cap V_{\varepsilon})\setminus E_{\delta}}\omega_{\varepsilon}^p\wedge\omega^{n-p}\geq\delta_p(U)-2^{n-p}\delta.$$
On the other hand, by (\ref{mass-cen-3}), we have
$$\widetilde{\alpha_{\varepsilon}}^p\geq\frac{C_{\varepsilon}}{\lambda_{p+1}\cdots\lambda_n}\omega_{\varepsilon}^p\ge\frac{C_{\varepsilon}}{M/\delta}\omega_{\varepsilon}^p,~~~~on~~(U\cap V_{\varepsilon})\setminus E_{\delta}.$$
From this we infer
\begin{align}\label{mass-cen-6}
\nonumber\int_{U\cap V_{\varepsilon}}\widetilde{\alpha_{\varepsilon}}^p\wedge\omega^{n-p}
&\geq\frac{C_{\varepsilon}}{M/\delta}\int_{(U\cap V_{\varepsilon})\setminus E_{\delta}}\omega_{\varepsilon}^p\wedge\omega^{n-p}\\
&\geq\frac{C_{\varepsilon}}{M/\delta}(\delta_p(U)-2^{n-p}\delta)>0.
\end{align}
provided that $\delta$ is taken small enough, e.g.,$\delta=2^{-(n-p+1)}\delta_p(U).$ \\
\textbf{Step 3.} The family of $(p,p)$-forms is uniformly bounded in mass since
$$\int_X\widetilde{\alpha_{\varepsilon}}^p\wedge\omega^{n-p}=\int_X\alpha^p\wedge\omega^{n-p}+O(\varepsilon)\le M$$
by the bound mass property. Let $T$ be any weak limit of $\widetilde{\alpha_{\varepsilon}}^p.$ By (\ref{mass-cen-6}), $T$ carries nonzero mass on every $p$-codimensional component of $Y$ (near regular point). By Skoda's extension theorem, $1_YT$ is a closed positive current with support in $Y$, hence $1_YT=\sum c_j[Y_j]$ is a combination of the various components $Y_j$ of $Y$ with coefficients $c_j>0$. This shows that $T$ belongs to the cohomology class $\{\alpha\}^p.$ Argue as the proof of Theorem 4.7 in \cite{Bou02} (cf. \cite[Lemma 4.10]{Bou02}), we can finish the proof of Proposition \ref{DP-conj}.
\end{proof}

\textbf{{Acknowledgments:}}  The author would like to thank Professor Zhiwei Wang and Zhuo Liu for useful discussions and insightful suggestions, which helped greatly to improve the exposition of the present article. The author is greatly indebted to Professor Zhenlei Zhang for his stimulating conversations and encouragements towards this problem.


\vspace{15mm}
\end{document}